\documentclass[12pt]{article}
\usepackage{amsfonts,amsthm}
\usepackage{amssymb,amsmath}
\usepackage{enumerate}
\usepackage{amsmath}

\usepackage[dvipsnames]{xcolor}
\usepackage{hyperref}
\hypersetup{
    colorlinks=true,
    citecolor=blue,
    linkcolor=blue,
    filecolor=magenta,
    urlcolor=blue,
    pdftitle={Overleaf Example},
    pdfpagemode=FullScreen,
    }

\input{amssym.def}

\newtheorem{Definition}{Definition}[section]
\newtheorem{Theorem}[Definition]{Theorem}
\newtheorem{Lemma}[Definition]{Lemma}
\newtheorem{Proposition}[Definition]{Proposition}
\newtheorem{Corollary}[Definition]{Corollary}
\newtheorem{Example}[Definition]{Example}
\newtheorem{Remark}[Definition]{Remark}

\newcommand{\be}{\begin{equation}}
\newcommand{\ee}{\end{equation}}

\begin{document}

\title{\bf On strongly $m$-$\Delta$-clean ring }
\author{\bf Saikat Das \footnote {e-mail : saikatofficial607@gmail.com} \ and ~\bf Sukhendu Kar  \footnote {e-mail : karsukhendu@yahoo.co.in} \\
{\small Department of Mathematics, Jadavpur University}\\
{\small 188, Raja S. C. Mallick Road, Kolkata - 700032, India.}}
\date{}
\maketitle

\begin{abstract}
Motivated by the recent study \cite{d1} of strongly $\Delta$-clean rings, we introduce and study strongly $m$-$\Delta$-clean rings. We establish their fundamental properties and characterize strongly $m$-$\Delta$-clean rings by lifting $m$-potent elements modulo $\Delta(R)$. We prove that every $m$-$\Delta$-clean ring is clean and that the factor ring of a strongly $m$-$\Delta$-clean ring modulo its Jacobson radical is reduced. Finally, we investigate corner rings, Morita context rings, and the relationships of these rings with $\Delta$-clean rings, local rings, semipotent rings, and strongly $m$-nil clean rings.\end{abstract}



\section{Introduction} \label{intro}
Throughout this paper, we assume that all rings are associative with identity. Let $R$ be a ring. Suppose that $U(R)$, $Id(R)$, $J(R)$ and $Nil(R)$ are the group of units, set of idempotents, Jacobson radical and set of nilpotent elements of $R$, respectively. Recall that
\[
\Delta(R)=\{x\in R \ : \ x+u\in U(R)\text{ for all }u\in U(R)\}.
\]

The subset $\Delta(R)$ was considered by Lam \cite[Exercise 4.24]{l0} and was later investigated in detail by Leroy and Matczuk \cite{l1}. From \cite[Exercise 4.24]{l0}, it follows that
\begin{align*}
\Delta(R) & =\{x\in R \ : \ 1-ux\in U(R)\text{ for all }u\in U(R)\} \\
& =\{x\in R \ : \ 1-xu\in U(R)\text{ for all }u\in U(R)\}.
\end{align*}

 Note that $\Delta(R)$ is a subring of $R$ satisfying $J(R)\subseteq \Delta(R)$, and it is the largest Jacobson radical subring of $R$ which is closed under multiplication by units. In general, $\Delta(R)$ need not be an ideal of $R$.

The study of decomposition properties of ring elements has played an important role in modern ring theory. The notion of a clean ring was introduced by Nicholson \cite{n1}. A ring $R$ is called clean if every element of $R$ can be expressed as the sum of an idempotent and a unit. Since then, numerous generalizations of clean rings have been investigated, including strongly clean rings, strongly nil-clean rings, $m$-clean rings, and strongly $m$-nil clean rings \cite{k2,b1,p1,d2,d3}. These classes have attracted considerable attention due to their rich structural properties and their close connections with exchange rings, lifting problems and decomposition theory.

Motivated by the role of $\Delta(R)$, several authors have investigated decompositions involving elements of $\Delta(R)$. In \cite{k1}, Karabacak, Ko\c{s}an, Quynh, and Tai studied $\Delta U$-rings, while Danchev, Hasanzadeh, Javan, and Moussavi \cite{d1} introduced strongly $\Delta$-clean rings. A ring $R$ is called strongly $\Delta$-clean ring if every element can be expressed as the sum of an idempotent and an element of $\Delta(R)$ that commute with each other. The authors studied several fundamental properties of strongly $\Delta$-clean rings.

In this paper, we introduce and investigate strongly $m$-$\Delta$-clean ring by replacing idempotent element with $m$-potent element. The notion of strongly $m$-$\Delta$-clean ring naturally generalizes the notion of strongly $\Delta$-clean ring and provide a unified framework for studying decompositions involving $m$-potent element and the subring $\Delta(R)$.

In Section~$2$, we establish the fundamental properties of $m$-$\Delta$-clean and strongly $m$-$\Delta$-clean rings. We prove that every $m$-$\Delta$-clean ring is clean. We further prove that if $R$ is strongly $m$-$\Delta$-clean, then nilpotent elements belong to $\Delta(R)$ and the factor ring $R/J(R)$ is reduced. In addition, we investigate annihilator conditions and show that strongly $m$-$\Delta$-cleanness is inherited by corner rings.

In Section~$3$, we obtain several characterizations for strongly $m$-$\Delta$-clean rings. In particular, we characterize these rings in terms of lifting $m$-potent elements modulo $\Delta(R)$. We further investigate the relationships among $m$-$\Delta$-clean rings, local rings, semipotent rings, $\Delta$-clean rings and strongly $m$-nil clean rings. Finally, we study the behavior of strongly $m$-$\Delta$-clean rings in the setting of Morita context rings.

\section{On $m$-$\Delta$-clean rings and strongly $m$-$\Delta$-clean rings}
We begin by recalling some basic properties of $\Delta(R)$ from \cite[Lemma 1.1]{l1}.

\begin{Proposition} \cite[Lemma 1.1]{l1} \label{6p2.1}
For any ring $R$,

$(i)$ $\Delta(R)=\{r\in R \ : \  ru+1\in U(R) \ \mbox{for all} \ {u\in U(R)} \}$ 

$\hspace{.68in}  =  \{r\ \in R \ : \  ur+1\in U(R) \ \mbox{for all} \ {u\in U(R)}\}$.

$(ii)$ for any $r\in \Delta(R)$ and $u\in U(R)$, we have $ur,\,ru\in \Delta(R)$.

$(iii)$ $\Delta(R)$ is a subring of $R$.

$(vi)$ $\Delta(R)$ is an ideal of $R$ if and only if $\Delta(R)=J(R)$.

$(v)$ If $R =\displaystyle \prod_{i \in I} R_i$, then $\Delta(R) =\displaystyle \prod_{i \in I} \Delta(R_i)$.
\end{Proposition}

\begin{Definition}
Let $m \geq 2$ and $R$ be a ring. An element $r \in R$ is called $m$-$\Delta$-clean if $r = f + d$, where $f$ is $m$-potent ($f^m = f$) and $d \in \Delta(R)$. 

A ring $R$ is called $m$-$\Delta$-clean if every element of $R$ is $m$-$\Delta$-clean.
\end{Definition}

\begin{Definition}
Let $m \geq 2$ and $R$ be a ring. An element $r \in R$ is called strongly $m$-$\Delta$-clean if $r = f + d$, where $f$ is $m$-potent ($f^m = f$), $d \in \Delta(R)$ and $fd = df$. 

A ring $R$ is said to be strongly $m$-$\Delta$-clean if every element of $R$ is strongly $m$-$\Delta$-clean.
\end{Definition}

Observe that (strongly) $2$-$\Delta$-clean rings are precisely (strongly) $\Delta$-clean rings. Furthermore, every (strongly) $\Delta$-clean ring is (strongly) $m$-$\Delta$-clean ring for every integer $m \geq 2$. However, the converse does not hold, in general.

\begin{Example}
 For any prime $p$, $\mathbb{Z}_p$ is a strongly $m$-$\Delta$-clean ring but it is not a strongly $\Delta$-clean ring.\end{Example}

\begin{Lemma} \cite[Lemma 2.2]{d1} \label{ls2.5}
For any ring $R$, $U(R) + \Delta(R)= U(R)$
\end{Lemma}

\begin{Proposition} \label{dp2.5}
Every $m$-$\Delta$-clean ring is a clean ring.
\end{Proposition}

\begin{proof}
Let $r \in R$. Since $R$ is a $m$-$\Delta$-clean ring, we have $r=f+d$, where $f^m = f$ and $d \in \Delta(R)$. Then $r=(1 - f^{m-1})+\big(f+f^{m-1}-1+d\big)$. Since $(1-f^{m-1})^2=1-f^{m-1}$, it follows that the element $1 - f^{m-1}$ is idempotent. Furthermore, $(f+f^{m-1}-1)(f^{m-2}-f^{m-1}+1)=1$ implies that $f+f^{m-1}-1 \in U(R)$. Since $d \in \Delta(R)$, it follows from Lemma \ref{ls2.5} that $f + f^{m-1} - 1 + d \in U(R) + \Delta(R) = U(R)$. Therefore, $r = e + u$, where $e = 1 - f^{m-1}$ is an idempotent and $u = f + f^{m-1} - 1 + d$ is a unit. Thus every element of $R$ is clean and hence $R$ is a clean ring.
\end{proof}

Now we can easily prove the following result : 
\begin{Proposition}
Let $R = \displaystyle\prod_{i \in I} R_i$ be a ring. Then $R$ is (strongly) $m$-$\Delta$-clean if and only if each $R_i$ is (strongly) $m$-$\Delta$-clean.
\end{Proposition}

\begin{Theorem}\cite[Proposition 1.6.]{l1}
Let $I$ be an ideal of $R$ such that $I \subseteq J(R)$. Then $\Delta(R/I) = \Delta(R)/I$.  
\end{Theorem}

\begin{Proposition} \label{pd2.8}
Let $R$ be a (strongly) $m$-$\Delta$-clean ring such that $I \subseteq J(R)$. Then $R/I$ is a (strongly) $m$-$\Delta$-clean ring.    
\end{Proposition}

\begin{proof}
Let $x \in R$. Then $x=f+d$ for some $m$-potent $f$ and $d \in \Delta(R)$. Therefore, $x+I=(f+I)+(d+I)$. Since $f$ is $m$-potent in $R$, it follows that $f+I$ is also an $m$-potent in $R/I$ and $d \in \Delta(R)$ implies that $d+I \in \Delta(R/I)=\Delta(R)/I$. Thus $x+I$ is $\Delta$-clean in $R/I$. Consequently, $R/I$ is a $m$-$\Delta$-clean ring. 

The strongly $m$-$\Delta$-clean case follows
similarly.
\end{proof}

From \cite{k1}, we recall that a ring $R$ is called a $\Delta U$-ring if $1+\Delta(R)=U(R)$. For further details on $\Delta U$-rings, we refer to \cite{k1}.

\begin{Proposition}\label{p2.9}
Let $R$ be a $\Delta U$-ring and $f^m = f$, $a \in \Delta(R)$. Then

$(i)$ $f - f^{m-1} + a \in \Delta(R)$,

$(ii)$ $f + f^{m-1} + a \in \Delta(R)$,

$(iii)$ $2f + a \in \Delta(R)$.
\end{Proposition}

\begin{proof}
Since $R$ is a $\Delta U$-ring, so $U(R) = 1 + \Delta(R)$.

$(i)$ Let $u = f-f^{m-1}+1$. Then $u$ is a unit in $R$. Thus there exists $d \in \Delta(R)$ such that $u=1+d$. Hence $f - f^{m-1} = d \in \Delta(R)$. Since  $\Delta(R)$ is a subring of $R$ and $a\in \Delta(R)$, it follows that $f-f^{m-1}+a \in \Delta(R)$.

$(ii)$ Since $f^m=f$, the element $f^{m-1}$ is idempotent. Consequently, $(2f^{m-1}-1)^2=1$ and $1-2f^{m-1} \in U(R)$. Therefore, there exists $d \in \Delta(R)$ such that $1-2f^{m-1}=1+d$, which yields $-2f^{m-1}\in \Delta(R)$. Since $\Delta(R)$ is a subring of $R$, $2f^{m-1} \in \Delta(R)$. Together with part $(i)$, we obtain
\[
f + f^{m-1} = (f - f^{m-1}) + 2f^{m-1} \in \Delta(R).
\] 
Hence $f + f^{m-1} + a \in \Delta(R)$.

$(iii)$ From $(i)$ and $(ii)$, it follows that
\[ 2f = (f + f^{m-1}) + (f - f^{m-1}) \in \Delta(R).\]
Since $a \in \Delta(R)$, we find that $2f + a \in \Delta(R)$. This completes the proof.
\end{proof}

\begin{Proposition} \label{prop2.10}
Let $R$ be a ring. If $f \in R$ be an $m$-potent and $d \in \Delta(R)$, then the following hold :

$(i)$ $(f \pm f^{m-1})d$, $d(f \pm f^{m-1}) \in \Delta(R)$

$(ii)$ For every positive integer $k$, $2f^{k}d \in \Delta(R)$. In particular, if $2$ is invertible in $R$, then $\Delta(R)$ is closed under multiplication by $m$-potent elements.
\end{Proposition}

\begin{proof}
$(i)$ Since $f$ is an $m$-potent element of $R$, we find that 
\[(1-f^{m-1}-f)(1-f^{m-1}-f^{m-2})=1=(1-f^{m-1}+f)(1-f^{m-1}+f^{m-2})\]

This implies that $f+f^{m-1}-1$ and $f-f^{m-1}+1$ are units in $R$. By Proposition \ref{6p2.1} $(ii)$, $\Delta(R)$ is closed under multiplication by units, it follows that $(f + f^{m-1} - 1)d \in \Delta(R)$ and $(f - f^{m-1} + 1)d \in \Delta(R)$.
Since $\Delta(R)$ is a subring of $R$ and $d \in \Delta(R)$, it follows that $(f \pm f^{m-1})d \in \Delta(R)$. 

Similarly, we can prove that $d(f\pm f^{m-1}) \in \Delta(R)$.

$(ii)$ From part $(i)$, we have $(f+f^{m-1})d \in \Delta(R)$ and $(f-f^{m-1})d \in \Delta(R)$. So we find that
$2fd=(f+f^{m-1})d+(f-f^{m-1})d \in \Delta(R)$.

Now let $k$ be any positive integer. Then $f^{k}$ is again an $m$-potent element. Applying the above argument to $f^{k}$, we conclude that
$2f^{k}d \in \Delta(R)$.

Finally, if $2 \in U(R)$, then $f^{k}d=2^{-1}(2f^{k}d) \in \Delta(R)$. In particular, taking $k=1$, we obtain $fd \in \Delta(R)$. Therefore, $\Delta(R)$ is closed under multiplication by $m$-potent elements.
\end{proof}

As an immediate consequence of the above {Proposition \ref{prop2.10}}, we can prove \cite[Proposition 2.6]{d1} by taking $m=2$ which is as follows :

\begin{Corollary}\cite[Proposition 2.6]{d1}
For every $e \in \mathrm{Id}(R)$ and $d \in \Delta(R)$, the containment $2ed \in \Delta(R)$ holds.
\end{Corollary}

\begin{Proposition} \label{prop2.9}
Let $R$ be a strongly $m$-$\Delta$-clean ring. If $a^2 \in \Delta(R)$, then $a \in \Delta(R)$.
\end{Proposition}
\begin{proof}
Let $a \in R$. Since $R$ is a strongly $m$-$\Delta$-clean ring, there exist $f, d \in R$ such that $a = f + d$,
where $f^m = f$, $d \in \Delta(R)$ and $fd = df$. Then $a^2 = (f + d)^2 = f^2 + 2fd + d^2$.

By assumption, $a^2 \in \Delta(R)$. Moreover, since $d \in \Delta(R)$, we have $d^2 \in \Delta(R)$. Also, since $f$ is $m$-potent and $d \in \Delta(R)$, it follows from the Proposition \ref{prop2.10} $(ii)$ that $2fd \in \Delta(R)$. Therefore, $f^2=a^2-2fd-d^2 \in \Delta(R)$. Next, $f^2=f\big(f-f^{m-1}+1\big)$ and $f-f^{m-1}+1$ is a unit in $R$. Therefore, multiplying by the inverse of $f-f^{m-1}+1$, we have $f=f^2(f-f^{m-1}+1)^{-1} \in \Delta(R)$. Consequently, $a = f + d \in \Delta(R)$. This completes the proof. \end{proof}

\begin{Corollary}\label{cor2.11}
Let $R$ be a strongly $m$-$\Delta$-clean ring. If $a^{2^k} \in \Delta(R)$ for some positive integer $k$, then $a \in \Delta(R)$.
\end{Corollary}

\begin{proof}
Suppose that $a^{2^k} \in \Delta(R)$ for some positive integer $k$. Since $a^{2^k}=\big(a^{2^{k-1}}\big)^2$, from Proposition \ref{prop2.9}, we find that $a^{2^{k-1}} \in \Delta(R)$. Repeating this argument successively yields that $a\in\Delta(R)$. \end{proof}

\begin{Corollary}
Let $R$ be a strongly $m$-$\Delta$-clean ring. Then every nilpotent element of $R$ belongs to $\Delta(R)$  and hence $\mathrm{Nil}(R) \subseteq \Delta(R)$.
\end{Corollary}

\begin{proof}
Let $a \in \mathrm{Nil}(R)$. Then there exists a positive integer $n$ such that $a^n=0$. Choose a positive integer $k$ such that $2^k \geq n$. Since $a^n=0$, we have $a^{2^k} = 0 \in \Delta(R)$. Therefore, from Corollary \ref{cor2.11}, it follows that $a \in \Delta(R)$. Hence $\mathrm{Nil}(R) \subseteq \Delta(R)$.
\end{proof}

\begin{Proposition} \label{dp2.15}
Let $R$ be a strongly $m$-$\Delta$-clean ring with $2 \in U(R)$. Then for any $a \in R$, $a-a^m \in \Delta(R)$.
\end{Proposition}

\begin{proof}
Let $a \in R$. Then $a=f+d$, where $f^m=f$, $d \in \Delta(R)$ and $fd=df$. Since $2 \in U(R)$, from Proposition~\ref{prop2.10}, we find that $f^i d^j \in \Delta(R)$ for all positive integers $i$ and $j$. Hence $a^m=(f+d)^m=f+d'$ for some $d' \in \Delta(R)$. Therefore, $a-a^m=d-d' \in \Delta(R)$.
\end{proof}

Recall that a ring $R$ is called reduced if it contains no nonzero nilpotent elements i.e. $Nil(R)=(0)$.

\begin{Lemma} \label{lem2.16}
 Let $R$ be a strongly $m$-$\Delta$-clean ring. Then $R/J(R)$ is reduced. \end{Lemma}

\begin{proof}
Suppose that $x^2 \in J(R)$. Since $J(R) \subseteq \Delta(R)$, from Proposition \ref{prop2.9}, it follows that $x \in \Delta(R)$.
We next show that $x \in J(R)$. Let $r \in R$. Then $(1-rx)(1+xr)=1+xr-rx-rx^2r=(xr-rx)+(1-rx^2r)$. Since $x^2 \in J(R)$, we have $1-rx^2r \in U(R)$. Therefore, it suffices to prove that $xr-rx \in \Delta(R)$. Since $R$ is a strongly $m$-$\Delta$-clean ring, so $r=f+d$, where $f^m=f$, $d \in \Delta(R)$ and $fd=df$. Then $xr-rx=x(f+d)-(f+d)x=(xf-fx)+(xd-dx)$. Since $x,d \in \Delta(R)$, we have $xd-dx \in \Delta(R)$. Thus it remains to show that $xf-fx \in \Delta(R)$.

Now $[f^{m-1}x(1-f^{m-1})]^2=0 \in \Delta(R)$, because $(1-f^{m-1})f^{m-1}=0$. From Proposition \ref{prop2.9}, we find that $f^{m-1}x-f^{m-1}xf^{m-1} \in \Delta(R)$. Similarly, $[(1-f^{m-1})xf^{m-1}]^2=0$ implies that $xf^{m-1}-f^{m-1}xf^{m-1} \in \Delta(R)$. Furthermore, $xf^{m-1}-f^{m-1}x=(xf^{m-1}-f^{m-1}xf^{m-1})-(f^{m-1}x-f^{m-1}xf^{m-1}) \in \Delta(R)$. Hence $xf-fx=f^{m-1}x-xf^{m-1}+x(f+f^{m-1})-(f+f^{m-1})x \in \Delta(R)$.
Consequently, $xr-rx \in \Delta(R)$. Therefore, $(1-rx)(1+xr)\in U(R)+\Delta(R)\subseteq U(R)$ by Lemma \ref{ls2.5}. Hence $1-rx$ is right invertible for every $r \in R$. This shows that $x \in J(R)$. 

Let $a+J(R)$ be a nilpotent element of $R/J(R)$. Then $a^k \in J(R)$ for some positive integer $k$. By the preceding observation if $x^2 \in J(R)$, we have $x \in J(R)$. Therefore $a^k \in J(R)$ implies $a \in J(R)$. Hence, $a+J(R)=J(R)$, the zero element of $R/J(R)$. Thus $R/J(R)$ is reduced.
\end{proof}

Let $R$ be a ring and $a \in R$. We denote 
\[
ann_l(a)=\{ r \in R : ra=0 \} \quad \text{and} \quad ann_r(a)=\{ r \in R : ar=0 \}.
\]

\begin{Lemma}\label{lem2.17}
Let $R$ be a ring and $a\in R$. Suppose that $a=f+d$ is a strongly $m$-$\Delta$-clean decomposition of $a$, where $f$ is an $m$-potent element, $d\in\Delta(R)$ and $fd=df$. Then
\[
ann_l(a)\subseteq ann_l(f)
\quad\text{and}\quad
ann_r(a)\subseteq ann_r(f).
\]
\end{Lemma}

\begin{proof}
We prove that $ann_l(a) \subseteq ann_l(f)$. The proof for the right annihilator is similar.

Let $r\in ann_l(a)$. Then $ra=0$ and hence $r(f+d)=0$. This implies that $rf=-rd$. Since $f^m=f$, multiplying the above equality on the right by $f^{m-1}$, we get $rf=-rdf^{m-1}$. Since $fd=df$, we have $rf=(-1)^t rd^t f^{\,m-t}$,
\ $ 1\le t\le m-1$. In particular, taking $t=m-1$, we have $rf=(-1)^{m-1}rd^{\,m-1}f$. Again, since $fd=df$, it follows that $rd^{\,m-1}f=rfd^{\,m-1}$. Therefore, $rf=(-1)^{m-1}rfd^{\,m-1}$. Hence $rf\bigl(1+(-1)^m d^{\,m-1}\bigr)=0$. Since $d\in\Delta(R)$, we have $(-1)^m d^{\,m-1}\in\Delta(R)$ and hence $1+(-1)^m d^{\,m-1}\in U(R)$. Therefore, $rf=0$. Thus $r \in ann_l(f)$ and consequently $ann_l(a) \subseteq ann_l(f)$. \end{proof}

\begin{Theorem}
Let $e$ be an idempotent in $R$. If $a \in eRe$ is strongly $m$-$\Delta$-clean in $R$, then $a$ is strongly $m$-$\Delta$-clean in $eRe$.
\end{Theorem}

\begin{proof}
Since $a$ is strongly $m$-$\Delta$-clean in $R$, there exist elements $f,d \in R$ such that $a=f+d$, where $f^m=f$, $d \in \Delta(R)$ and $fd=df$. As $a \in eRe$, we have $a=ea=ae$. Hence $(1-e)a=a(1-e)=0$. This implies that $1-e \in ann_l(a)\cap ann_r(a)$. By Lemma \ref{lem2.17}, it follows that $1-e \in ann_l(f)\cap ann_r(f)$. Therefore, $(1-e)f=f(1-e)=0$, which yields $f=ef=fe$. Now $a=f+d=efe+ede$. Since $f^m=f$ and $f=ef=fe$, we find that $(efe)^m=efe$. So $efe$ is an $m$-potent element of $eRe$. Thus it is enough to show that $ede \in \Delta(eRe)$. Now we prove that $eRe \cap \Delta(R) \subseteq \Delta(eRe)$. Let $r \in eRe \cap \Delta(R)$. Take any unit $u \in U(eRe)$. Then $(u+(1-e))(u^{-1}+(1-e))=1$. This implies that $u+(1-e) \in U(R)$. Since $r \in \Delta(R)$, it follows that $1-(u+(1-e))r \in U(R)$. Consequently, there exists an element $v \in R$ such that $(1-(u+(1-e))r)v=1$. Since $r \in eRe$, we have $(u+(1-e))r=ur$ and hence $(1-ur)v=1$. Now multiply the above equality by $e$ on both sides, we get $e(1-ur)ve=e$. Since $u,r \in eRe$, we obtain $(e-ur)eve=e$. Observe that $eve \in eRe$. Thus $e-ur$ is left invertible in $eRe$ for every $u \in U(eRe)$. Consequently, it follows that $r \in \Delta(eRe)$.
Hence $eRe \cap \Delta(R) \subseteq \Delta(eRe)$. Furthermore $ede=a-efe=a-f=d \in eRe \cap \Delta(R) \subseteq \Delta(eRe)$. Therefore, $d \in \Delta(eRe)$. Moreover, $fd=df$ and $f,d \in eRe$. Hence $a=efe+ede$ is a strongly $m$-$\Delta$-clean decomposition of $a$ in $eRe$. This shows that $a$ is strongly $m$-$\Delta$-clean in $eRe$.
\end{proof}

\begin{Corollary} \label{cor2.20}
Let $e$ be an idempotent in $R$. If $R$ is a strongly $m$-$\Delta$-clean ring, then $eRe$ is also a strongly $m$-$\Delta$-clean ring.   
\end{Corollary}

\section{Characterizations of strongly $m$-$\Delta$-clean rings}
We say that $m$-potent element lifts modulo $\Delta(R)$ if, whenever $a\in R$ satisfies $a-a^{m}\in\Delta(R)$,
there exists an $m$-potent element $f\in R$ such that $a-f\in\Delta(R)$. If, in addition, $af=fa$, then we say that $m$-potent element lifts strongly modulo $\Delta(R)$.

\begin{Theorem}
Let $R$ be a ring with $2 \in U(R)$. Then the followings are equivalent :

$(i)$ $R$ is a strongly $m$-$\Delta$-clean ring.

$(ii)$ For any $x \in R$, $x-x^m \in \Delta(R)$ and each $m$-potent lifts strongly modulo $\Delta(R)$.
\end{Theorem}

\begin{proof}
$(i) \Longrightarrow (ii)$. 
Assume that $R$ is a strongly $m$-$\Delta$-clean ring and let $x \in R$. Then by Proposition \ref{dp2.15}, we have $x - x^m \in \Delta(R)$. Now let $a \in R$ be such that $a - a^m \in \Delta(R)$. Since $R$ is strongly $m$-$\Delta$-clean, there exist an $m$-potent element $f \in R$ and an element $d \in \Delta(R)$ such that $a = f + d$ and $af = fa$. Thus it follows that $a - f = d \in \Delta(R)$, where $f$ is $m$-potent and commutes with $a$. Therefore, every $m$-potent element lifts strongly modulo $\Delta(R)$.

$(ii) \Longrightarrow (i)$. 
Let $x \in R$. Then $x - x^m \in \Delta(R)$. Since each $m$-potent element lifts strongly modulo $\Delta(R)$, there exists an $m$-potent element $f \in R$ such that $x - f \in \Delta(R)$ and $xf = fx$. Set $d = x - f \in \Delta(R)$. Then $x = f + d$, where $f$ is $m$-potent and $d \in \Delta(R)$. Moreover, from $xf = fx$ it follows that $df = fd$. Hence $R$ is a strongly $m$-$\Delta$-clean ring.
\end{proof}

\begin{Proposition}
The following conditions are equivalent for a ring $R$:

$(i)$ $R$ is a $\Delta U$-ring.

$(ii)$ All $m$-clean elements of $R$ are $\Delta$-clean.
\end{Proposition}

\begin{proof}
 \noindent
$(i) \Longrightarrow (ii)$. 
Let $r\in R$ be an $m$-clean element. Then $r=f+u$, where $f^m=f$ and $u\in U(R)$. Since $R$ is a $\Delta U$-ring, so $u=1+d$ for some $d\in\Delta(R)$. Hence $r=(1-f^{m-1})+(f+f^{m-1}+d)$. Since $(1-f^{m-1})^2=1-f^{m-1}$, the element $1-f^{m-1}$ is idempotent. Also from Proposition~\ref{p2.9}, we find that $f+f^{m-1}+d\in\Delta(R)$.
Thus $r$ is $\Delta$-clean.

\noindent
$(ii) \Longrightarrow (i)$. Let $u\in U(R)$. Then $u$ is $m$-clean and hence $u$ is $\Delta$-clean. Therefore, $u=e+d$, where $e^2=e$ and $d\in\Delta(R)$. Thus it follows that $eu^{-1}=1-du^{-1}$. By Proposition \ref{6p2.1} $(ii)$ $du^{-1}\in\Delta(R)$, then we have $-du^{-1}\in\Delta(R)$ and hence $1-du^{-1}\in 1+\Delta(R)\subseteq U(R)$. Therefore, $eu^{-1} \in U(R)$. This implies that $e \in U(R)$. Since $e$ is an idempotent as well as a unit, so $e = 1$. Consequently, $u = 1 + d \in 1 + \Delta(R)$. Thus $U(R) \subseteq 1 + \Delta(R)$. The reverse inclusion is immediate and hence $U(R) = 1 + \Delta(R)$. Therefore, $R$ is a $\Delta U$-ring.   
\end{proof}

A ring $R$ is called an \emph{$m$-potent ring} if every element of $R$ is $m$-potent, i.e., $a^m=a$ for all $a\in R$. 

\begin{Theorem} \label{thm3.4}
A ring $R$ is an $m$-potent ring if and only if $R$ is a strongly $m$-$\Delta$-clean ring and $\Delta(R)=0$.
\end{Theorem}

\begin{proof}
Let $R$ be an $m$-potent ring and let $x\in\Delta(R)$. Since $\Delta(R)$ is a subring of $R$, we have $x^{m-1}\in\Delta(R)$. By the definition of $\Delta(R)$, $1 - x^{m-1}$ is a unit in $R$. Moreover, $x^m=x$ and hence $x(1 - x^{m-1}) = 0$. Since $1-x^{m-1}$ is invertible, it follows that $x=0$. Thus $\Delta(R) = 0$ . Since every element of $R$ is $m$-potent and $0\in\Delta(R)$, $R$ is a strongly $m$-$\Delta$-clean ring.

Conversely, suppose that $R$ is a strongly $m$-$\Delta$-clean ring and $\Delta(R)=0$. Then for any $a\in R$, we have $a=f+d$, where $f^m=f$ and $d\in\Delta(R)$. Thus $d=0$ as $\Delta(R)=0$. Therefore, every element of $R$ is $m$-potent and hence $R$ is an $m$-potent ring. \end{proof}

A ring $R$ is called a semipotent ring if each one-sided ideal $I$ not contained in $J(R)$ contains a nonzero idempotent. A ring $R$ is called local if it has a unique maximal ideal.

\begin{Theorem} \label{dthm2.21}
Let $R$ be a ring with only trivial idempotents $0$ and $1$. Then $R$ is a semipotent ring if and only if $R$ is a local ring.
\end{Theorem}

\begin{proof}
Suppose that $R$ is a semipotent ring. Let $x \in R$ be such that $x \notin U(R)$. We claim that $x \in J(R)$. If possible, let $x \notin J(R)$. Then the principal ideal $(x)$ is not contained in $J(R)$. Since $R$ is a semipotent ring, there exists a nonzero idempotent $e \in (x)$. By hypothesis, the only idempotents in $R$ are $0$ and $1$. So $e = 1$. Hence $1 \in (x)$. Thus there exists $t \in R$ such that $1 = tx$. This shows that $x$ is a unit, contradicting the assumption that $x \notin U(R)$. Therefore, $x \in J(R)$ and hence every non-unit of $R$ lies in $J(R)$. Consequently, $R$ is a local ring.

Conversely, suppose that $R$ is a local ring. Then $R = U(R) \cup J(R)$. Let $I$ be an ideal of $R$ such that $I \nsubseteq J(R)$. Then there exists $x \in I$ such that $x \notin J(R)$. Since $R$ is a local ring, this implies that $x \in U(R)$. So $1 \in I$ and hence $I = R$. In particular, $I$ contains the idempotent $1 \neq 0$. Thus every ideal of $R$ not contained in $J(R)$ contains a nonzero idempotent. Therefore, $R$ is a semipotent ring.
\end{proof}

\begin{Theorem} \label{dthm2.20}
Let $R$ be a ring with only trivial idempotents $0$ and $1$. If $R$ is a $m$-$\Delta$-clean ring, then $R$ is a local ring.
\end{Theorem}

\begin{proof}
Let $x\in R$. Then $x=f+d$, where $f^m=f$ and $d\in\Delta(R)$. Thus $f^{m-1}$ is an idempotent. As $R$ has only trivial idempotents, either $f^{m-1}=0$ or $f^{m-1}=1$. If $f^{m-1}=0$, then $f=f^m=f^{m-1}f=0$ and hence $x=d\in\Delta(R)$. Therefore, $1-x\in U(R)$. If $f^{m-1}=1$, then $f\in U(R)$. Hence $x=f+d\in U(R)+\Delta(R)=U(R)$ by Lemma \ref{ls2.5}. Thus for every $x\in R$, either $x\in U(R)$ or $1-x\in U(R)$. Hence $R$ is a local ring.
\end{proof}

The following result follows immediately from Theorem~\ref{dthm2.21}.
\begin{Corollary}
Let $R$ be a ring with only trivial idempotents $0$ and $1$. If $R$ is an $m$-$\Delta$-clean ring, then $R$ is a semipotent ring.
\end{Corollary}

\begin{Remark}
The converse of Theorem~\ref{dthm2.20} does not hold, in general. Indeed, the ring $\mathbb{Z}_3$ has only trivial idempotents and $\mathbb{Z}_3$ is also a local ring. However, $\mathbb{Z}_3$ is not $2$-$\Delta$-clean, since $2$ cannot be expressed as the sum of an idempotent element and an element of $\Delta(\mathbb{Z}_3)$. Thus, $2$ is not a $2$-$\Delta$-clean element of $\mathbb{Z}_3$.
\end{Remark}

For clarity, we use the notation $U_{m-1}$ to represent the set of $(m-1)$-th roots of unity which is defined as follows : $U_{m-1}(R) = \{ u \in R : u^{m-1} = 1 \}$.

\begin{Theorem}
Let $R$ be a ring and $U_{m-1}(R) \subseteq 1+\Delta(R)$. Then $R$ is a strongly $m$-$\Delta$-clean ring if and only if $R$ is a strongly $\Delta$-clean ring.\end{Theorem}

\begin{proof}
Suppose that $R$ is a strongly $m$-$\Delta$-clean ring. Then $a=f+d$, where $f^m=f$, $d\in\Delta(R)$ and $fd=df$. Consider the element $u=f-f^{m-1}+1$. Then $u^{m-1}=1$ and hence $u\in U_{m-1}(R)$. Since $U_{m-1}(R)\subseteq1+\Delta(R)$, so $u=1+d'$ for some $d'\in\Delta(R)$. Therefore, $f=f^{m-1}+d'$. Consequently, $a=f^{m-1}+(d+d')$. Since $d+d'\in\Delta(R)$, $f^{m-1}$ is an idempotent and $f^{m-1}(d+d')=(d+d')f^{m-1}$,
it follows that $a$ is a strongly $\Delta$-clean element in $R$. Hence $R$ is a strongly $\Delta$-clean ring.

Conversely, suppose that $R$ is a strongly $\Delta$-clean ring. Since every idempotent is also $m$-potent, so every strongly $\Delta$-clean decomposition of an element in $R$ is also a strongly $m$-$\Delta$-clean decomposition. Hence $R$ is a strongly $m$-$\Delta$-clean ring.
\end{proof}  

Let $A$ and $B$ be rings, and let $M$ and $N$ be $(A,B)$- and $(B,A)$-bimodules, respectively. A \emph{Morita context} is a $6$-tuple $(A,B,M,N,\psi,\varphi)$, where
$\psi:M\times N\longrightarrow A$ and $\varphi:N\times M\longrightarrow B$ are additive pairings satisfying the usual associativity conditions. The set $R=
\begin{pmatrix}
A&M\\
N&B
\end{pmatrix}$ forms a ring under the usual matrix addition and multiplication induced by the bimodule actions and the pairings $\psi$ and $\varphi$. This ring is called the \emph{Morita context ring}.

\begin{Lemma} \label{lem3.9}
Let $R=\begin{pmatrix} A & M \\ N & B \end{pmatrix}$ be the Morita context. If $R$ is a strongly $m$-$\Delta$-clean ring, then $A$ and $B$ are strongly $m$-$\Delta$-clean rings with $MN \subseteq J(A)$ and $NM \subseteq J(B)$. \end{Lemma}

\begin{proof}
Consider the idempotent $e=\begin{pmatrix}
1 & 0\\
0 & 0
\end{pmatrix}\in R$. Then $eRe\cong A$ and $(1-e)R(1-e)\cong B$. Since $R$ is a strongly $m$-$\Delta$-clean ring, from Corollary \ref{cor2.20}, it follows that the corner rings $eRe$ and $(1-e)R(1-e)$ are also strongly $m$-$\Delta$-clean rings. Hence $A$ and $B$ are strongly $m$-$\Delta$-clean rings. Next from Lemma \ref{lem2.16}, it follows that every nilpotent element of a strongly $m$-$\Delta$-clean ring belongs to $J(R)$. Therefore, for every $m\in M$ and $n\in N$, we obtain
$\begin{pmatrix}
0 & m\\
0 & 0
\end{pmatrix},
\,
\begin{pmatrix}
0 & 0\\
n & 0
\end{pmatrix}
\in J(R)$. Therefore, it follows that
$\begin{pmatrix}
0 & m\\
n & 0
\end{pmatrix}
=
\begin{pmatrix}
0 & m\\
0 & 0
\end{pmatrix}
+
\begin{pmatrix}
0 & 0\\
n & 0
\end{pmatrix}
\in J(R)$
for every $m\in M$ and $n\in N$. Now let $m_1,m_2\in M$ and $n_1,n_2\in N$. Then 
$\begin{pmatrix}
m_1n_1 & 0\\
0 & n_2m_2
\end{pmatrix}
=
\begin{pmatrix}
0 & m_1\\
n_2 & 0
\end{pmatrix}
\begin{pmatrix}
0 & m_2\\
n_1 & 0
\end{pmatrix}
\in J(R)$. Consequently,
$\begin{pmatrix}
MN & M\\
N & NM
\end{pmatrix}
\subseteq J(R)$.
Finally, by~\cite[Theorem~1]{s}, it follows that $MN\subseteq J(A)$ and $NM\subseteq J(B)$. This completes the proof.
\end{proof}    

From \cite{b1}, we recall that a ring $R$ is called \emph{strongly $m$-nil clean} if every element $a\in R$ can be expressed as $a=f+n$, where $f^m=f$, $n\in Nil(R)$ and $fn=nf$. The following is an immediate consequence of \cite[Theorem 1.10]{b1}, by considering conditions $(1)$ and $(2)$ of that theorem.

\begin{Proposition} \cite[Theorem 1.10]{b1} \label{mnp3.10}
Let $R$ be a ring and $m \geq 2$ be an integer with $m \not\equiv 1~(mod~3)$, $m \not\equiv 1~(mod~8)$. Then the following conditions are equivalent :

$(i)$ $R$ is strongly $m$-nil clean.

$(ii)$ $J(R)$ is nil and $R/J(R)$ is $m$-potent.
\end{Proposition}

\begin{Theorem} \label{lem3.10}
Let $R$ be a ring and $m \geq 2$ be an integer with $m \not\equiv 1~(mod~3)$, $m \not\equiv 1~(mod~8)$. Then $R$ is a strongly $m$-nil clean ring if and only if $R$ is a strongly $m$-$\Delta$-clean ring and $\Delta(R) \subseteq Nil(R)$.     
\end{Theorem}
\begin{proof}
 We first show that every element of $\Delta(R)$ is nilpotent. Let $a \in \Delta(R)$. Since $R$ is a strongly $m$-nil clean ring, there exist an $m$-potent element $f$ and a nilpotent element $n$ in $R$ such that $a=f+n$ with $fn=nf$. Consider the element $1-f^{m-1}=\bigl(f+(1-f^{m-1})\bigr)-f=v-f$, where $v=f+(1-f^{m-1})$ is a unit in $R$. Since $\bigl(f+f^{m-1}-1\bigr)\bigl(f^{m-2}-f^{m-1}+1\bigr)=1$, we have $1-f^{m-1}=v-f=v-(a-n)=(v-a)+n$. Since $a\in\Delta(R)$ and $v\in U(R)$, we have
$u=v-a\in U(R)+\Delta(R)=U(R)$ by Lemma \ref{ls2.5}. Therefore, $1-f^{m-1}=u+n=u\bigl(1+u^{-1}n\bigr)\in U(R)$. On the other hand, $1-f^{m-1}$ is an idempotent. Hence $1-f^{m-1}=1$. This implies that $f^{m-1}=0$. Therefore, $f=f^{m-1}f=0$. Thus $a=n\in Nil(R)$. Consequently, each element of $\Delta(R)$ is nilpotent. On the other hand, by Proposition \ref{mnp3.10}, $R/J(R)$ is $m$-potent ring. Then $Nil(R) \subseteq J(R)$. Since $r \in Nil(R)$, so $r^k=0$ for some integer $k$. Then $(r+J(R))^k=J(R)$. Again $R/J(R)$ is an $m$-potent ring. This implies that $r+J(R)=J(R)$ which shows that $r \in J(R)$.
Thus it follows that $Nil(R) \subseteq J(R) \subseteq \Delta(R)$. Hence $Nil(R)=\Delta(R)$. Hence every strongly $m$-nil clean decomposition is a strongly $m$-$\Delta$-clean decomposition. Therefore, $R$ is a strongly $m$-$\Delta$-clean ring and $\Delta(R)$ is nil.

Conversely, suppose that $a\in R$. Since $R$ is a strongly $m$-$\Delta$-clean ring, there exist an $m$-potent element $f\in R$ and an element $d\in\Delta(R)$ such that
$a=f+d$ with $fd=df$. By the assumption $\Delta(R)\subseteq Nil(R)$, it follows that $d$ is nilpotent. Hence
$R$ is a strongly $m$-nil clean ring. \end{proof}

\begin{Theorem}
Let $R=\begin{pmatrix} A & M \\ N & B \end{pmatrix}$ be a Morita context and $m\ge 2$ be an integer satisfying $m \not\equiv 1~(mod~3)$ and $m \not\equiv 1~(mod~8)$. Suppose that $J(A)$ and $J(B)$ are nilpotent and $\Delta(A/J(A))=\Delta(B/J(B))=0$. Then the following are equivalent :

$(i)$ $R$ is a strongly $m$-$\Delta$-clean ring.

$(ii)$ $A$ and $B$ are strongly $m$-$\Delta$-clean rings with $MN \subseteq J(A)$ and $NM \subseteq J(B)$.
\end{Theorem}

\begin{proof}
$(i) \Longrightarrow (ii)$. This implication follows immediately from Lemma \ref{lem3.9}.

$(ii) \Longrightarrow (i)$. Assume that $A$ and $B$ are strongly $m$-$\Delta$-clean rings satisfying $MN\subseteq J(A)$ and $NM\subseteq J(B)$. By Proposition \ref{pd2.8}, the factor rings $A/J(A)$ and $B/J(B)$ are strongly $m$-$\Delta$-clean rings. Since $\Delta(A/J(A))=\Delta(B/J(B))=0$, from Theorem~\ref{thm3.4}, it follows that $A/J(A)$ and $B/J(B)$ are $m$-potent rings. As $J(A)$ and $J(B)$ are nilpotent ideals, they are nil ideals. Hence from Proposition \ref{mnp3.10}, we find that both $A$ and $B$ are strongly $m$-nil clean rings. Furthermore, since $MN\subseteq J(A)$, $NM\subseteq J(B)$ and both $J(A)$ and $J(B)$ are nilpotent ideals, it follows that $MN$ and $NM$ are nilpotent. Applying \cite[Theorem 4.7]{b1}, we conclude that the Morita context ring $R=\begin{pmatrix}
A & M\\
N & B
\end{pmatrix}$ is a strongly $m$-nil clean ring. Finally, by Theorem \ref{lem3.10}, every strongly $m$-nil clean ring is strongly $m$-$\Delta$-clean ring. Therefore, $R$ is a strongly $m$-$\Delta$-clean ring.
\end{proof}

\vspace{0.5cm}

\noindent
\textbf{Acknowledgment}

\noindent
The first author acknowledges the financial support received as a Senior Research Fellow (SRF) from the University Grants Commission (UGC), Government of India (Award Letter No. 211610081351).

\noindent
\textbf{Declarations} 

\noindent
\textbf{Conflict of interest:} All authors declare that they have no conflicts of interest.

\end{document}